\documentclass[11pt]{amsart}

\title[Positive flow-spines and contact $3$-manifolds, II]{Positive flow-spines and contact $3$-manifolds, II}

\author[Ishii]{Ippei Ishii}
\address{Department of Mathematics, Faculty of Science and Technology, Keio University, 3-14-1 Hiyoshi, Kohoku, Yokohama 223-8522, Japan}

\author[Ishikawa]{Masaharu Ishikawa}
\address{Department of Mathematics, Hiyoshi Campus, Keio University, 
4-1-1 Hiyoshi, Kohoku, Yokohama 223-8521, Japan}
\email{ishikawa@keio.jp}

\author[Koda]{Yuya Koda}

\address{Department of Mathematics, Hiyoshi Campus, Keio University, 
4-1-1 Hiyoshi, Kohoku, Yokohama 223-8521, Japan / International Institute for Sustainability with Knotted Chiral Meta Matter, Hiroshima University, 1-7-1 Kagamiyama, Higashi-Hiroshima, 739-8526, Japan}
\email{koda@keio.jp}

\author[Naoe]{Hironobu Naoe}
\address{Department of Mathematics, Chuo University, 
1-13-27 Kasuga, Bunkyo-ku, Tokyo, 112-8551, Japan}
\email{naoe@math.chuo-u.ac.jp}

\usepackage{amsfonts,amsmath,amssymb,amscd}
\usepackage{amsthm}
\usepackage{latexsym}
\usepackage{graphicx}
\usepackage{psfrag}
\usepackage{color}

\usepackage{ulem}

\usepackage{fancybox, ascmac}
\usepackage{multicol}

\theoremstyle{plain}
\newtheorem*{theorem*}{Theorem}
\newtheorem*{lemma*} {Lemma}
\newtheorem*{corollary*} {Corollary}
\newtheorem*{proposition*}{Proposition}
\newtheorem*{conjecture*}{Conjecture}
\newtheorem{theorem}{Theorem}[section]

\newtheorem{corollary}[theorem]{Corollary}
\newtheorem{proposition}[theorem]{Proposition}

\theoremstyle{remark}

\newtheorem*{definition}{Definition}

\newtheorem*{example*}{Example}

\theoremstyle{definition}

\newtheoremstyle{citing}
  {}
  {}
  {\itshape}
  {}
  {\bfseries}
  {.}
  {.5em}
  {\thmnote{#3}}

\theoremstyle{citing}

\newcommand{\Int}{\mathrm{Int\,}}

\newcommand{\Nbd}{\mathrm{Nbd}}

\newcommand{\pr}{\mathrm{pr}}

\makeatletter
\@addtoreset{equation}{section}

\makeatother

\begin{document}

\begin{abstract}
This paper corresponds to Section 8 of arXiv:1912.05774v3 [math.GT]. The contents until Section~7 are published in Annali di Matematica Pura ed Applicata 
as a separate paper.
In that paper, it is proved that for any positive flow-spine $P$ of a closed, oriented $3$-manifold $M$, there exists a unique contact structure supported by $P$ up to isotopy.
In particular, this defines 
a map from the set of isotopy classes of positive flow-spines of $M$ to the set of isotopy classes of contact structures on $M$.
In this paper, we show that this map is surjective.
As a corollary, we show that any flow-spine can be deformed to a positive flow-spine by applying first and second regular moves successively.
\end{abstract}

\maketitle


\section{Introduction}

In~\cite{IIKN1}, closed, connected, oriented, contact $3$-manifolds are related to $2$-dimensional simple polyhedra with specific structure, called {\it positive flow-spines}. 
A non-singular flow $\mathcal F$ in a closed, connected, oriented, $3$-manifold $M$
is said to be {\it carried by a flow-spine $P$} of $M$ if $\mathcal F$ is transverse to $P$ and 
is a ``constant flow'' in the complement of $P$ in $M$.
A contact structure $\xi$ on $M$ is said to be {\it supported} by a flow-spine $P$
if there exists a contact form $\alpha$ on $M$ such that $\xi=\ker\alpha$ and 
its Reeb flow is carried by $P$.
A flow-spine has two types of vertices: the vertex on the left in Figure~\ref{fig3} is said to be {\it of $\ell$-type} and the one on the right is {\it of $r$-type}.
A flow-spine is then said to be {\it positive} if it has at least one vertex and all vertices are of $\ell$-type.
It is proved in~\cite{IIKN1} that for any positive flow-spine $P$ of a closed, oriented $3$-manifold $M$, there  exists a unique contact structure supported by $P$ up to isotopy.
This defines a map from the set of isotopy classes of positive flow-spines of $M$ to the set of isotopy classes of contact structures on $M$.
The following main theorem of the present paper shows that this map is surjective.

\begin{figure}[htbp]
\begin{center}
\includegraphics[width=7.5cm, bb=146 625 446 712]{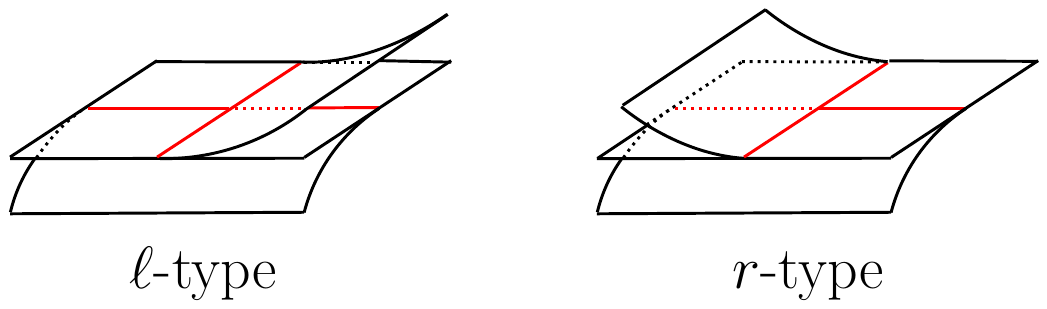}
\caption{Vertices of $\ell$-type and $r$-type. Here, the ambient space is equipped with the right-handed orientation.}\label{fig3}
\end{center}
\end{figure}




\begin{theorem}\label{thm02}
For any closed, connected, oriented, contact $3$-manifold $(M,\xi)$, there exists a positive flow-spine $P$ of $M$ that supports $\xi$.
\end{theorem}

As mentioned above, this theorem gives a surjection from the set of isotopy classes of positive flow-spines of $M$ to the set of isotopy classes of contact structures on $M$.
It also gives a surjection from the set of positive flow-spines up to homeomorphism to the set of contact $3$-manifolds up to contactomorphism.
To get a one-to-one correspondence, we need to find suitable moves of positive flow-spines. Moves of flow-spines are known in~\cite{Ish92, BP97, EI05}. 
However, moves of positive flow-spines 
are very rare since most known moves for flow-spines yield $r$-type vertices. 
Theorem~\ref{thm02} implies that any flow-spine can be deformed to a positive one by regular moves, see Corollary~\ref{cor02}.

The surjection from the set of positive flow-spines up to homeomorphism to the set of contact $3$-manifolds up to contactomorphism allows us to define a complexity for contact $3$-manifolds like the Matveev complexity for usual $3$-manifolds. This will be discussed in a forthcoming paper.

In Section~2, we briefly introduce fundamental notions such as contact $3$-manifolds, branched polyhedra, flow-spines, regular moves and the construction of a contact form from a positive flow-spine used in~\cite{IIKN1}.
The proof of Theorem~\ref{thm02} is given in Section~3.

The second author is supported by JSPS KAKENHI Grant Numbers JP19K03499 and JSPS-VAST Joint Research Program, Grant number JPJSBP120219602.
The third author is supported by JSPS KAKENHI Grant Numbers
JP20K03588 and JP21H00978.
The fourth author is supported by JSPS KAKENHI Grant Numbers JP19K21019 and 20K14316.


\section{Preliminaries}

Throughout this paper, for a polyhedral space $X$, $\Int X$ represents the interior of $X$,
$\partial X$ represents the boundary of $X$, and $\Nbd(Y;X)$ represents a closed regular neighborhood of a subspace $Y$ of $X$ in $X$,
where $X$ is equipped with the natural PL structure if $X$ is a smooth manifold.
The set $\Int\Nbd(Y;X)$ is the interior of $\Nbd(Y;X)$ in $X$.

\subsection{Contact $3$-manifolds}\label{sec1}

In this subsection, we briefly recall notions and known results in $3$-dimensional contact topology
that will be used in this paper.
The reader may find general explanation, for instance, in~\cite{Etn06, Gei08, OS04}.

Let $M$ be a closed, oriented, smooth $3$-manifold. A {\it contact structure} on $M$ is the $2$-plane field on $M$ given by the kernel of a $1$-form $\alpha$ on $M$ satisfying $\alpha\land d\alpha\ne 0$ everywhere. The $1$-form $\alpha$ is called a {\it contact form}.
If $\alpha\land d\alpha>0$ everywhere on $M$ then the contact structure given by $\ker\alpha$ is called a {\it positive contact structure}
and the $1$-form $\alpha$ is called a {\it positive contact form}.
The pair of a closed, oriented, smooth $3$-manifold $M$ and a contact structure $\xi=\ker\alpha$ on $M$ is called a {\it contact $3$-manifold} and denoted by $(M,\xi)$. In this paper, by a contact structure we mean a positive one.

Two contact structures $\xi$ and $\xi'$ on $M$ are said to be {\it isotopic} if there exists a one-parameter family of contact forms $\alpha_t$, $t\in [0,1]$, such that 
$\xi=\ker\alpha_0$ and $\xi'=\ker\alpha_1$. 
Two contact $3$-manifolds $(M,\xi)$ and $(M',\xi')$
are said to be {\it contactomorphic} if there exists a diffeomorphism $\phi:M\to M'$ such that $\phi_*(\xi)=\xi'$. The map $\phi$ is called a {\it contactomorphism}. 
If $M=M'$ then we also say that $\xi$ and $\xi'$ are contactomorphic. The Gray theorem states that if two contact structures are isotopic then they are contactomorphic~\cite{Gra59}.

Next we introduce the Reeb vector field. Let $\alpha$ be a contact form on $M$.
A vector field $X$ on $M$ determined by the conditions $d\alpha(X,\cdot)=0$ and $\alpha(X)=1$ is called the {\it Reeb vector field} of $\alpha$ on $M$.
Such a vector field is uniquely determined by $\alpha$ and we denote it by $R_\alpha$.
The non-singular flow on a $3$-manifold $M$ generated by a Reeb vector field is called a {\it Reeb flow}.

The Reeb vector field plays important roles in many studies in contact geometry and topology.
In $3$-dimensional contact topology, it is used to give a correspondence between contact structures and open book decompositions of $3$-manifolds.
Let $\Sigma$ be an oriented, compact surface with boundary and $\phi:\Sigma\to\Sigma$ be a diffeomorphism such that $\phi|_{\partial \Sigma}$ is the identity map on $\partial \Sigma$. If a closed, oriented $3$-manifold $M$ is orientation-preservingly homeomorphic to the quotient space
obtained from $\Sigma\times [0,1]$ by the identification $(x,1)\sim (\phi(x),0)$ for $x\in \Sigma$ and $(y,0)\sim (y,t)$ for each $y\in \partial \Sigma$ and any $t\in [0,1]$,
then we say that it is an open book decomposition of $M$.
The image $L$ of $\partial \Sigma$ in $M$ by the quotient map is called the {\it binding},
which equips the orientation as the boundary of $\Sigma$.
The image of the surface $\Sigma\times\{t\}$ in $M$ is called a {\it page}.
We denote the open book by $(M, \Sigma, L, \phi)$.

A contact structure $\xi$ on $M$ is said to be {\it supported} by an open book $(M, \Sigma, L, \phi)$ if there exists a contact form $\alpha$ on $M$ such that $\xi=\ker\alpha$
and the Reeb vector field $R_\alpha$ of $\alpha$ satisfies that
\begin{itemize}
\item $R_\alpha$ is tangent to $L$ and the orientation on $L$ induced from $\Sigma$ coincides with the direction of $R_\alpha$, and
\item $R_\alpha$ is positively transverse to $\Int\Sigma\times\{t\}$ for any $t\in [0,1]$.
\end{itemize}
Any open book has a supported contact structure and this is used by Thurston and Winkelnkemper to prove that any closed, oriented, smooth $3$-manifold admits a contact structure~\cite{TW75}. Note that the existence of a contact structure for any $3$-manifold was first proved by Martinet~\cite{Mar71}.
Giroux then showed that the contact structure supported by a given open book is unique up to isotopy~\cite{Gir02}. He also proved the following theorem, that will be used in the proof of Theorem~\ref{thm02}.

\begin{theorem}[Giroux, cf.~\cite{Etn06}]\label{thm_giroux}
For any contact $3$-manifold $(M,\xi)$, there exists an open book decomposition of $M$ that supports $\xi$.
\end{theorem}

These results of Giroux give a surjection
from the set of isotopy classes of open book decompositions of $M$
to the set of isotopy classes of contact structures on $M$, and also 
a surjection from the set of open books up to homeomorphism to 
the set of contact $3$-manifolds up to contactomorphism.
Taking the quotient of the set of open books of $M$ by a certain operation for pages of open books, called a {\it stabilization}, we can think one-to-one correspondence between 
open books and contact $3$-manifolds, that is called the {\it Giroux correspondence}.


\subsection{Branched polyhedron}

A compact topological space $P$ is called a {\it simple polyhedron}, or a 
{\it quasi-standard polyhedron}, 
if every point of $P$ has a regular neighborhood homeomorphic to 
one of the three local models shown in Figure~\ref{fig1}. 
A point whose regular neighborhood is shaped on the model 
(iii) is called a {\it true vertex} of $P$ (or {\it vertex} for short), and we denote the set of true vertices of $P$  by $V(P)$. 
The set of points whose regular neighborhoods are shaped on the models 
(ii), (iii) and (v) is called 
the {\it singular set} of $P$, and we denote it by $S(P)$. 
The set of points whose regular neighborhoods are shaped on the models (iv) and (v) is called 
the {\it boundary} of $P$, and we denote it by $\partial P$.
Each connected component of $P\setminus S(P)$ is called a {\it region} of $P$ and each connected component of $S(P)\setminus V(P)$ is called an {\it edge} of $P$.
A simple polyhedron $P$ is said to be {\it special}, or {\it standard},
if each region of $P$ is an open disk and each edge of $P$ is an open arc.
Throughout this paper, we assume that all regions are orientable.

\begin{figure}[htbp]
\begin{center}
\includegraphics[width=10.0cm, bb=129 531 543 712]{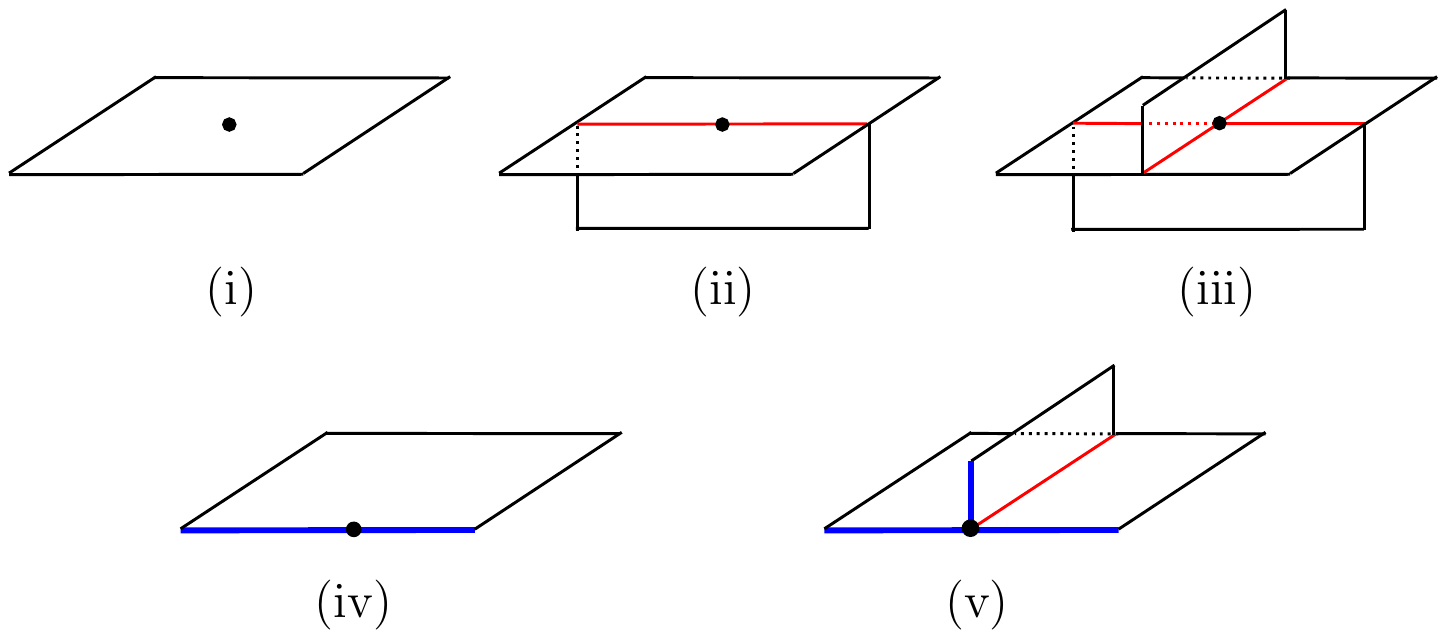}
\caption{The local models of a simple polyhedron.}\label{fig1}
\end{center}
\end{figure}

A {\it branching} of a simple polyhedron $P$ is an assignment of orientations to
regions of $P$ such that the three orientations on each edge of $P$ induced by the three adjacent regions do not agree.
We note that even though each region of a simple polyhedron $P$ is orientable, 
$P$ does not necessarily admit a branching. 
See~\cite{Ish86, BP97, Kod07, Pet12} for general properties of branched polyhedra.

\subsection{Flow-spines}\label{sec22}

A polyhedron $P$ is called a {\it spine} of a closed, connected, oriented $3$-manifold $M$ 
if it is embedded in $M$ and $M$ with removing an open ball collapses onto $P$.
If a spine is simple then it is called a {\it simple spine}.

If a simple spine $P$ admits a branching, then it allows us to smoothen $P$ in the ambient manifold $M$ as in the local models shown in Figure~\ref{fig2}. 
A point of $P$ whose regular neighborhood is shaped on the model~(iii) is called a {\it vertex of $\ell$-type} and that on the model~(iv) is a {\it vertex of $r$-type}.

\begin{figure}[htbp]
\begin{center}
\includegraphics[width=14.0cm, bb=129 649 587 712]{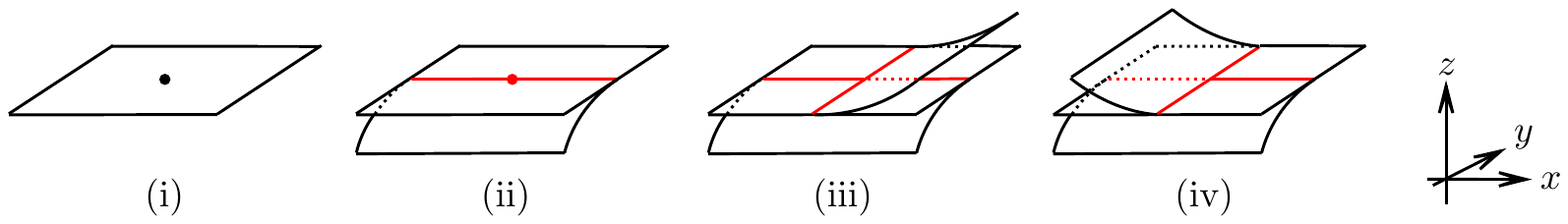}
\caption{The local models of a branched simple polyhedron.}\label{fig2}
\end{center}
\end{figure}

\begin{definition}
Let $M$ be a closed, connected, oriented $3$-manifold.
\begin{itemize}
\item[(1)] Let $(M,\mathcal F)$ be a pair of $M$ and a non-singular flow $\mathcal F$ on $M$. A simple spine $P$ of $M$ is called a {\it flow-spine} of $(M,\mathcal F)$ if, for each point $p\in P$, there exists a positive coordinate chart $(U; x,y,z)$ of $M$ around $p$ such that $(U,p)$ is one of the models in Figure~\ref{fig2}, where $\mathcal F|_U$ is generated by $\frac{\partial}{\partial z}$, and $\mathcal F|_{M\setminus P}$ is a constant vertical flow shown in Figure~\ref{fig2-2}.
\item[(2)] A branched simple spine $P$ of $M$ is called a {\it flow-spine} of $M$ if it is a flow-spine of $(M,\mathcal F)$ for some non-singular flow $\mathcal F$ on $M$.
The flow $\mathcal F$ is said to be {\it carried by $P$}.
\end{itemize}
\end{definition}


\begin{figure}[htbp]
\begin{center}
\includegraphics[width=4.5cm, bb=177 609 353 713]{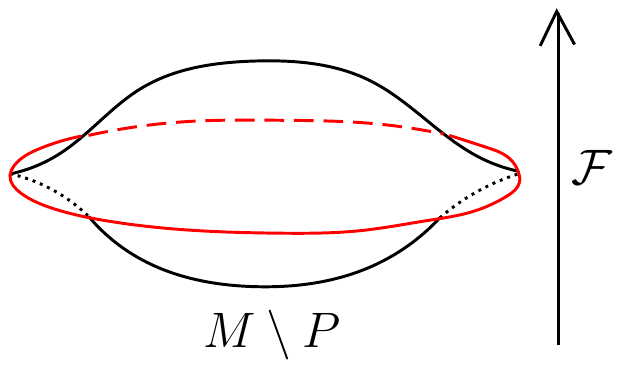}
\caption{The flow in the complement $M\setminus P$.}\label{fig2-2}
\end{center}
\end{figure}

\subsection{Regular moves}

In this subsection, we introduce two kinds of moves of branched simple polyhedra.

\begin{definition}
\begin{itemize}
\item[(1)] The moves shown in Figure~\ref{fig3-7} and their mirrors are called  {\it first regular moves}~\cite{Ish92}
or {\it Matveev-Piergallini moves}~\cite{BP97}.
\item[(2)] The moves shown in Figure~\ref{fig3-8} and their mirrors are called {\it second regular moves}~\cite{Ish92} or {\it lune moves}~\cite{Mat03}.
\end{itemize}
\end{definition}

\begin{figure}[htbp]
\begin{center}
\includegraphics[width=9cm, bb=129 526 536 713]{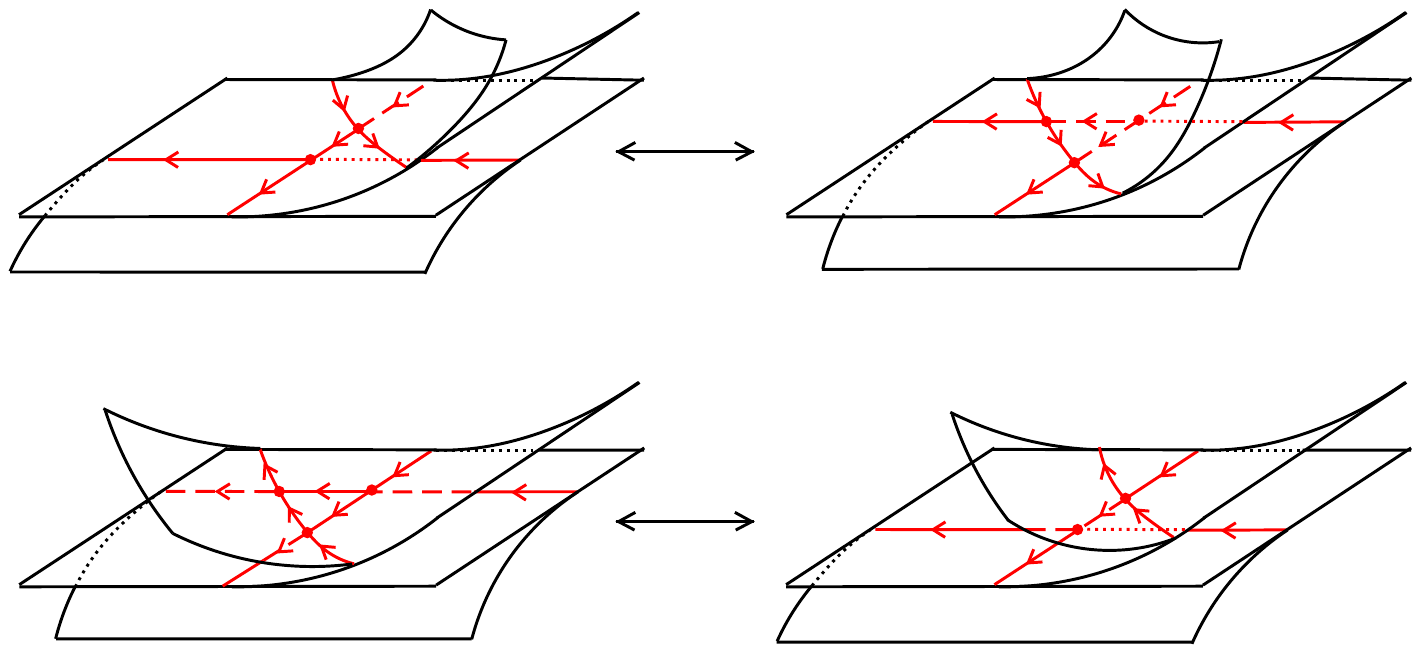}
\caption{First regular moves.}\label{fig3-7}
\end{center}
\end{figure}

\begin{figure}[htbp]
\begin{center}
\includegraphics[width=11cm, bb=129 605 475 712]{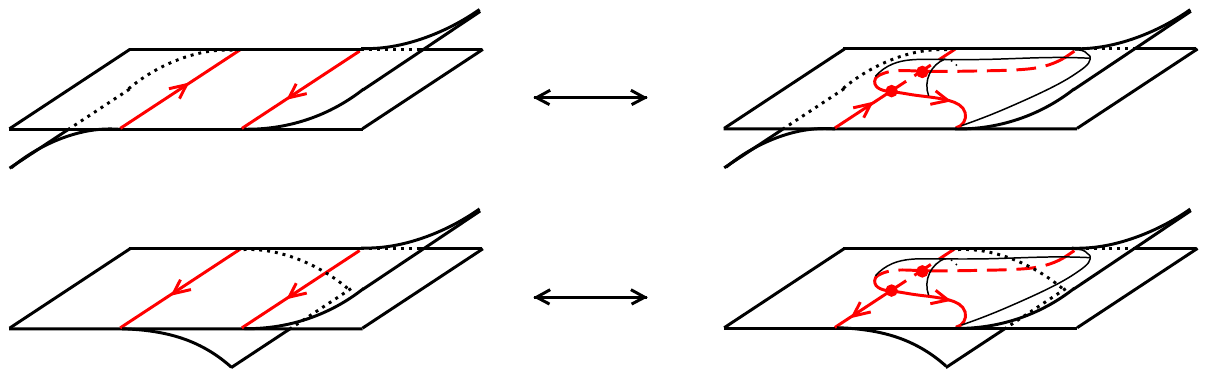}
\caption{Second regular moves.}\label{fig3-8}
\end{center}
\end{figure}

These moves are enough to trace deformation of non-singular flows in $3$-manifolds.
Two non-singular flows $\mathcal F_1$ and $\mathcal F_2$ in a closed, orientable $3$-manifold are said to be {\it homotopic} 
if there exists a smooth deformation from $\mathcal F_1$ to $\mathcal F_2$ in the set of non-singular flows.

\begin{theorem}[\cite{Ish92}]\label{thm_ishii}
Let $P_1$ and $P_2$ be flow-spines of a closed, orientable $3$-manifold $M$.
Let $\mathcal F_1$ and $\mathcal F_2$ be flows on $M$ carried by $P_1$ and $P_2$, respectively. 
Suppose that $\mathcal F_1$ and $\mathcal F_2$ are homotopic.
Then $P_1$ is obtained from $P_2$ by applying first and second regular moves successively.
\end{theorem}

\subsection{Construction of a contact form from a positive flow-spine}\label{sec25}

In this subsection, we briefly recall the construction of a contact form on 
a closed, connected, oriented, smooth $3$-manifold $M$ from its positive flow-spine $P$ introduced in~\cite{IIKN1}.

For a positive flow-spine $P$ of $M$, set $Q=\Nbd (S(P);P)$ and 
let $\mathcal R_i$ be the connected component of $P\setminus \Int Q$ contained in 
a region $R_i$ of $P$.
Let $\Nbd(S(P);M)$ be a neighborhood of $S(P)$ in $M$,
choose a neighborhood $N_P=\Nbd(P;M)$ of $P$ in $M$ sufficiently thin with respect to $\Nbd(S(P);M)$,
and set $N_Q=\Nbd(P;M)\cap \Nbd(S(P);M)$. 
Then $M$ decomposes as
\[
   M=N_P\cup N_D=\left(N_Q\cup\bigcup_{i=1}^n N_{\mathcal R_i}\right)\cup N_D,
\]
where $N_{\mathcal R_i}$ is the connected component of the closure of $N_P\setminus N_Q$ containing $\mathcal R_i$, 
which is diffeomorphic to $\mathcal R_i\times [0,1]$, and $N_D$ is the closure of $M\setminus N_P$, which is a $3$-ball.
The boundary of $N_P$ consists of ``vertical part'' and ``horizontal part''.
The ``vertical part'' is the vertical face of $N_P$ shown on the right in Figure~\ref{fig7}, which constitutes an annulus. The union of the rest faces of $N_P$ is the ``horizontal part''.
We regard $N_D$ as $D^2\times [0,1]$, where $D^2$ is the unit $2$-disk, so that $\partial D^2\times [0,1]$ corresponds to the ``vertical part'' of $\partial N_P$.

\begin{figure}[htbp]
\begin{center}
\includegraphics[width=13.2cm, bb=129 617 472 711]{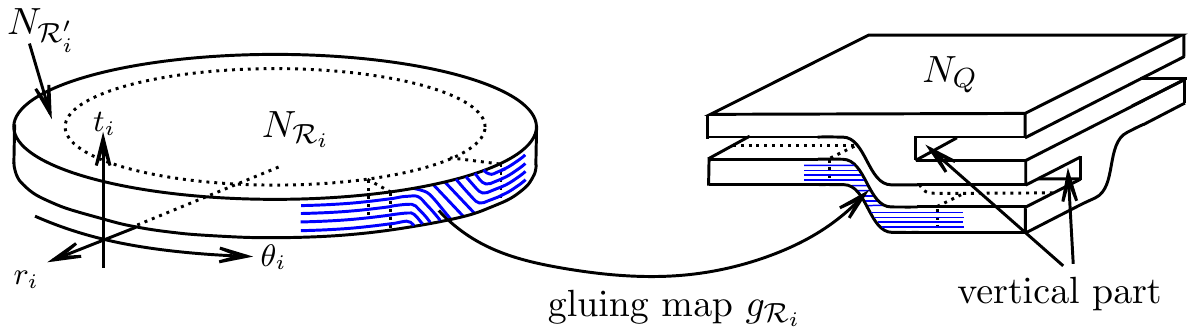}
\caption{Construction of $N_P$.}\label{fig7}
\end{center}
\end{figure}

A $1$-form $\eta$ on $M$ constructed in~\cite[Section 5]{IIKN1}, called a {\it reference $1$-form}, satisfies that
\begin{itemize}
\item[(a)] $\eta\land d\eta\geq 0$ on $M$ and
\item[(b)] $d\eta=0$ on $N_D\setminus (\hat W\cup\check W)$. 
\end{itemize}
Property~(a) is proved in \cite[Lemma 5.8]{IIKN1}.
Property~(b) follows from Lemma 5.6 (1), (2) and equation~(5.4) in~\cite{IIKN1}. 
We omit the definition of $\hat W\cup\check W$. This part appears in small neighborhoods of the orbits of the flow passing though the vertices of $P$, cf.~\cite[Fig. 18]{IIKN1}, and we only use this fact in the proof of Theorem~\ref{thm02}.

The next step is to construct a $1$-form $\beta$ on a positive flow-spine $P$ of $M$ such that $d\beta>0$.
Let $\pr_Q:Q\to Q'$ be a projection of $Q$
to an oriented, compact surface $Q'$ such that, for each edge $e$ of $P$, the two adjacent regions of $Q$ that induce the same orientation to $e$ are identified. 
Note that $Q'$ is homotopy equivalent to $S(P)$. 
A $1$-form on $P$ means that it is obtained by gluing the pullback of a $1$-form on $Q'$ 
by the projection $\pr_Q:Q\to Q'$ and $1$-forms on the regions $R_1,\ldots,R_n$ of $P$.
The existence of a $1$-form $\beta$ on $P$ with $d\beta>0$ is proved in~\cite[Lemma 6.3]{IIKN1}.

From this $\beta$, we make a $1$-form $\hat\beta$ on $N_P$.
Let $\pr_{\mathcal R_i}:N_{\mathcal R_i}=\mathcal R_i\times [0,1]\to \mathcal R_i$ be the first projection
and set $N_{\mathcal R_i'}=\Nbd(\partial\mathcal R_i;\mathcal R_i)\times [0,1]\subset 
\mathcal R_i\times [0,1]$. 
Let $g_{\mathcal R_i}:N_{\mathcal R_i'}\to N_Q$ be a gluing map of $N_{\mathcal R_i}$ to $N_Q$.
Choose coordinates $(r_i,\theta_i)$ on $\Nbd(\partial \mathcal R_i;\mathcal R_i)$ so that $\{r_i=0\}=\partial \Nbd(\partial \mathcal R_i;\mathcal R_i)\setminus \partial \mathcal R_i$, $\{r_i=1\}=\partial \mathcal R_i$ and the orientation of $(r_i,\theta_i)$ coincides with that of $\mathcal R_i$, see the left in Figure~\ref{fig7}.
Let $\pr_{N_Q}:N_Q\to Q$ be a canonical projection.
From the $1$-form $\beta$ on $P$,
we define the $1$-form $\hat\beta$ on $N_P$ by gluing $\pr_{N_Q}^*\beta$ on 
$N_Q\setminus\bigcup_{i=1}^ng_{\mathcal R_i}(N_{\mathcal R'_i})$ 
and
\[
   (1-\sigma(r_i))\pr_{\mathcal R_i}^*\beta+\sigma(r_i)(\pr_{N_Q}\circ g_{\mathcal R_i})^*\beta
\]
on $N_{\mathcal R_i}$\footnote{This is equation (6.1) in~\cite{IIKN1},
where a projection $\pr:N_Q\to Q'$ such that $\pr|_Q=\pr_Q$ is used instead of $\pr_{N_Q}$.
The positions of $(\pr\circ g_{\mathcal R_i})^*\beta$ and $\pr_{\mathcal R_i}^*\beta$ in equation (6.1) should be opposite. It does not affect the later discussion in~\cite{IIKN1}.}, 
where 
$\sigma(r_i)$ is a monotone increasing smooth function such that $\sigma(\varepsilon)=0$ and 
$\sigma(1-\varepsilon)=1$ for any sufficiently small $\varepsilon\geq 0$.
It is verified in~\cite[Lemma 6.6]{IIKN1} that $\hat\beta+R\eta$ is a contact form on $N_P$ whose Reeb vector field is positively transverse to $P$ for any $R>0$.

Next we extend the $1$-form $\hat\beta$ to $N_D=D^2\times [0,1]$.
We extend the coordinate $[0,1]$ slightly and set $N_D^\varepsilon=D^2\times [-\varepsilon,1+\varepsilon]$
for a sufficiently small $\varepsilon>0$.
Let $\hat\beta_E$ and $\hat\beta_F$ be the $1$-form $\hat\beta$ on $D^2\times [-\varepsilon,0]$ and 
$D^2\times [1,1+\varepsilon]$, respectively. Fix coordinates $(u,v,w)$ on $N_D^\varepsilon$ so that
the $1$-forms $\hat\beta_E$ and $\hat\beta_F$ are invariant on their defining domains under translation to the $w$-coordinate.
We then extend these forms to the whole $N_D$ canonically and denote them again by $\hat\beta_E$ and $\hat\beta_F$.
Now we define a $1$-form $\hat\beta_M$ on $M$ by extending $\hat\beta$ on $N_P$ to $N_D=D^2\times [0,1]$ as
\[
   \hat\beta_M=(1-\tau(w))\hat\beta_E+\tau(w)\hat\beta_F,
\]
where $\tau:[-\varepsilon,1+\varepsilon]\to [0,1]$ is a monotone increasing smooth function such that 
$\tau(w)=0$ for $-\varepsilon\leq w\leq 0$ and $\tau(w)=1$ for $1\leq w\leq 1+\varepsilon$.
The contact form on $M$ asserted in~\cite[Theorem 1.1]{IIKN1} is given as $\hat\beta_M+R\eta$
for a sufficiently large $R>0$.




\section{Proof of Theorem~\ref{thm02}}

In this section, we give the proof of Theorem~\ref{thm02}. 

\begin{definition}
A vector field $Z$ 
on $M$ is called a {\it monodromy vector field} of 
an 
open book $(M, \Sigma, L, \phi)$ if 
\begin{itemize}
\item $Z$ is tangent to the binding $L$ and the orientation on $L$ induced from $\Sigma$ coincides with the direction of $Z$, and
\item $Z$ is positively transverse to $\Int\Sigma\times\{t\}$ for any $t\in [0,1]$.
\end{itemize}
\end{definition}

In particular, the Reeb vector field used in the definition of the contact structure supported by an open book in Section~\ref{sec1} is a monodromy vector field of the open book. 

We first prove the following proposition.

\begin{proposition}\label{prop61}
Let $(M, \Sigma, L, \phi)$ be an open book decomposition of 
a closed, connected, oriented $3$-manifold $M$ whose page is a surface of genus at least $2$ and whose binding is connected.
Then there exists a positive flow-spine on $M$ that carries the flow generated by a monodromy vector field of the open book.
\end{proposition}

\begin{proof}
Let $g$ be the genus of $\Sigma$ and
$a_1,\ldots,a_g$, $b_1,\ldots,b_{g-1}$, $c_1,\ldots,c_g$ be the simple closed curves
on the page as in Figure~\ref{fig30}.
Let $\sigma_{a_i}$, $\sigma_{b_i}$ and $\sigma_{c_i}$ 
denote the right-handed Dehn twists along $a_i$, $b_i$ and $c_i$, respectively.

\begin{figure}[htbp]
\begin{center}
\includegraphics[width=11.0cm, bb=130 632 463 713]{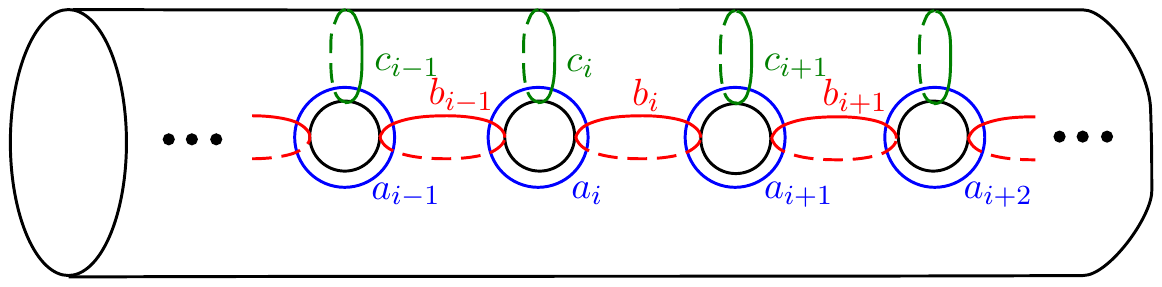}
\caption{Curves for Dehn twists.}\label{fig30}
\end{center}
\end{figure}

We represent the monodromy of the open book by a product of Dehn twists as
$w_1^{m_1}w_2^{m_2}\cdots w_n^{m_n}$,
where $w_k$ is the Dehn twist along one of the curves $a_1,\ldots,a_g$, $b_1,\ldots,b_{g-1}$, $c_1,\ldots,c_g$. Here we read the product from left to right.
Let $f:M\setminus \Int \Nbd(L;M)\to S^1$ be a smooth non-singular map of the open book. 
Choose $6n$ points $q_1,\ldots,q_{6n}$ on $S^1$ in the counter-clockwise order and set $\Sigma_j=f^{-1}(q_j)$.
We assign to each $\Sigma_j$ the orientation as a page of the open book.

We first draw a graph, which will become $S(P)$, on each $\Sigma_j$ as shown 
in Figures~\ref{fig31},~\ref{fig32}, \ref{fig33}, \ref{fig34}, \ref{fig35} and \ref{fig36}, 
depending on $j$ modulo $6$.

\begin{figure}[htbp]
\begin{center}
\includegraphics[width=10.0cm, bb=130 596 463 713]{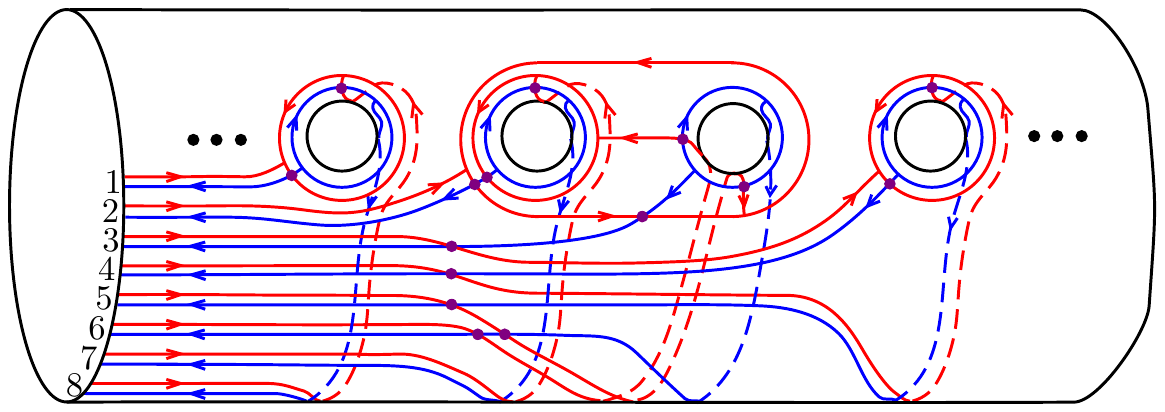}
\caption{The singular set $S(P)$ on $\Sigma_{6k}$.}\label{fig31}
\end{center}
\end{figure}

\begin{figure}[htbp]
\begin{center}
\includegraphics[width=10.0cm, bb=130 596 463 713]{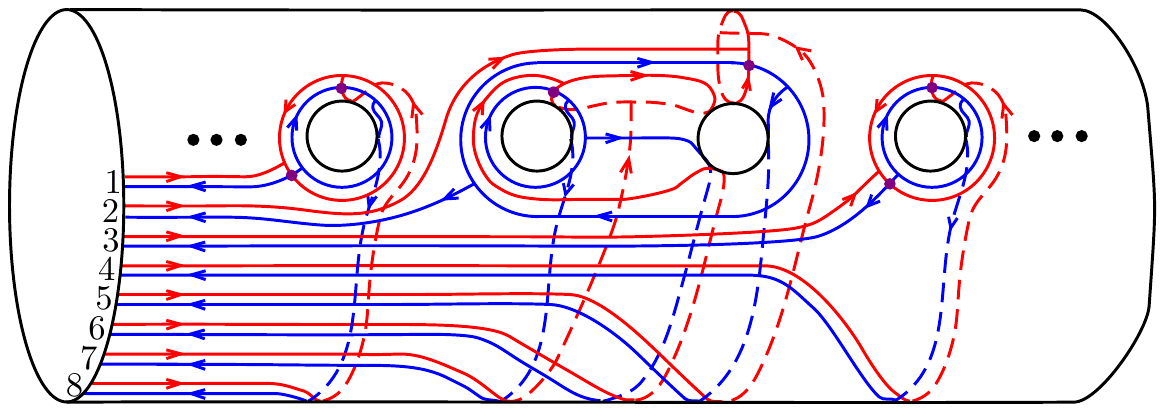}
\caption{The singular set $S(P)$ on $\Sigma_{6k+1}$.}\label{fig32}
\end{center}
\end{figure}

\begin{figure}[htbp]
\begin{center}
\includegraphics[width=10.0cm, bb=130 596 463 713]{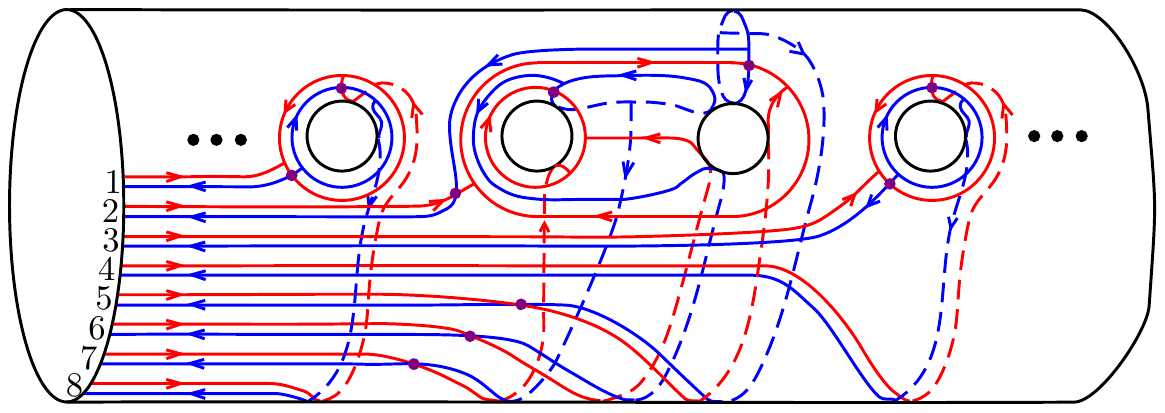}
\caption{The singular set $S(P)$ on $\Sigma_{6k+2}$.}\label{fig33}
\end{center}
\end{figure}

\begin{figure}[htbp]
\begin{center}
\includegraphics[width=10.0cm, bb=130 595 463 712]{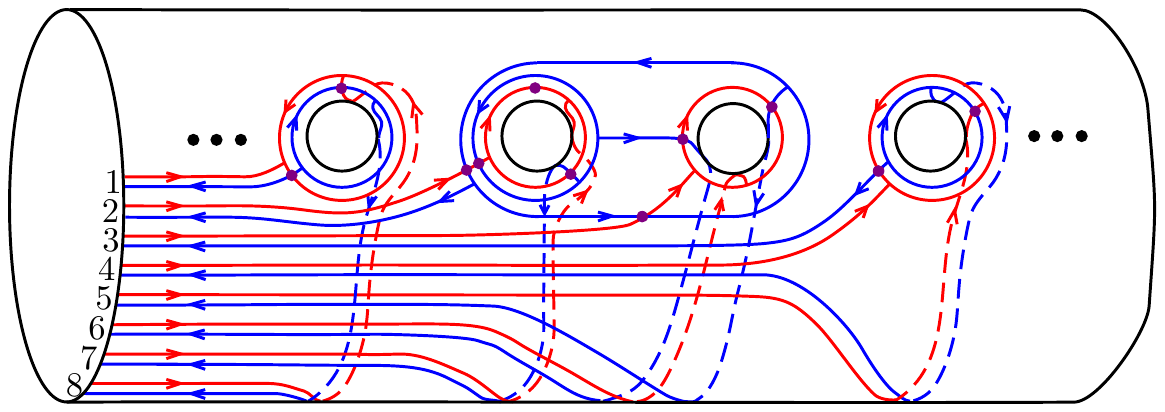}
\caption{The singular set $S(P)$ on $\Sigma_{6k+3}$.}\label{fig34}
\end{center}
\end{figure}

\begin{figure}[htbp]
\begin{center}
\includegraphics[width=10.0cm, bb=130 596 463 713]{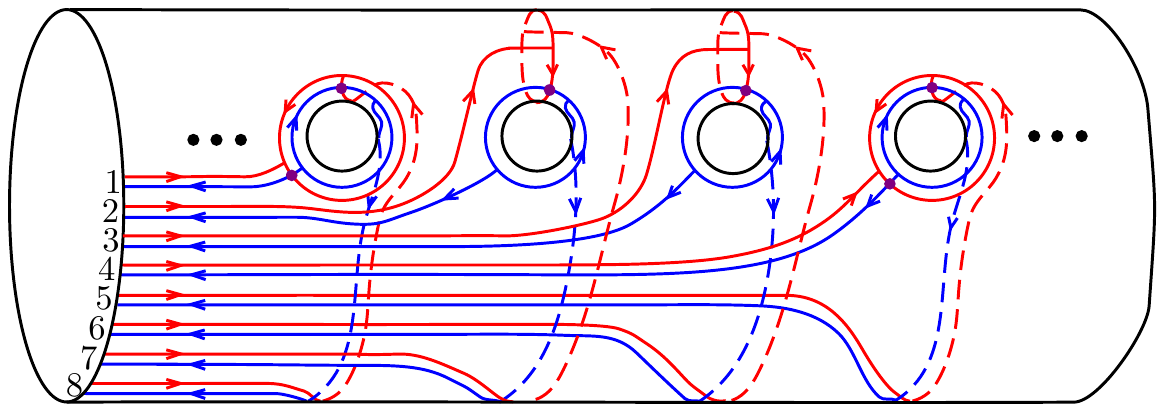}
\caption{The singular set $S(P)$ on $\Sigma_{6k+4}$.}\label{fig35}
\end{center}
\end{figure}

\begin{figure}[htbp]
\begin{center}
\includegraphics[width=10.0cm, bb=130 596 463 713]{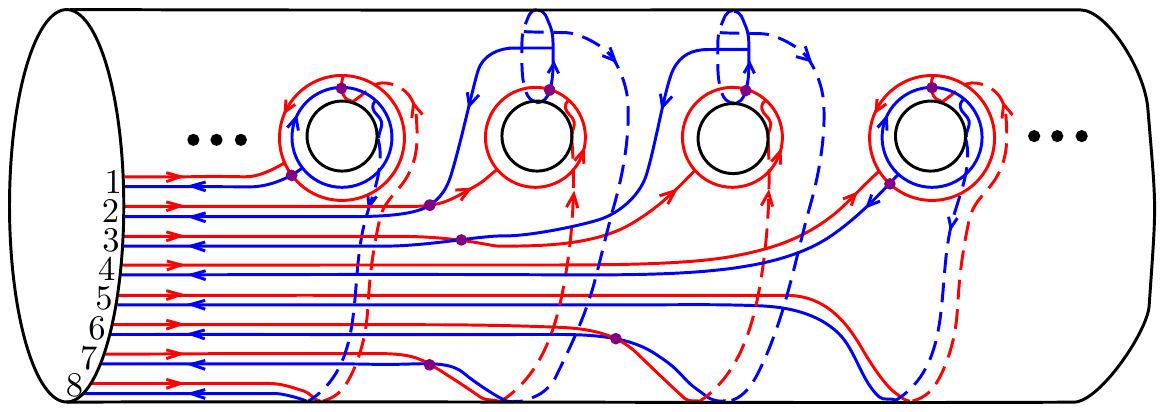}
\caption{The singular set $S(P)$ on $\Sigma_{6k+5}$.}\label{fig36}
\end{center}
\end{figure}

For each $j=1,\ldots,6n$, the red graph $G_j^{red}$ 
on $\Sigma_j$ and the blue graph $G_j^{blue}$ on $\Sigma_{j+1}$ coincide as unoriented graphs up to isotopy,
which we denote by $G_j$,
and their orientations are opposite. Here we regard $\Sigma_{6n+1}=\Sigma_1$.
We attach $G_j\times [0,1]$ between $\Sigma_j$ and $\Sigma_{j+1}$ such that
$G_j\times\{0\}=G_j^{red}$ and $G_j\times\{1\}=G_j^{blue}$.
To each region of $G_j\times[0,1]$ we assign an orientation so that the induced orientation on the boundary coincides with those of the edges of $G_j^{red}$ and $G_j^{blue}$.
See Figure~\ref{fig38} for the orientations on $G_j\times[0,1]$.
A vertex of this branched polyhedron can be seen as either an intersection of a red edge and a blue edge on $\Sigma_j$ or a trivalent vertex on $\Sigma_j$. 
Any vertex in the former case is of $\ell$-type since the blue edge intersects the red one from the left to the right with respect to the orientation of the red edge.
Other vertices are those appearing in Figure~\ref{fig38} and we can verify that they are of $\ell$-type in the figure directly.

\begin{figure}[htbp]
\begin{center}
\includegraphics[width=6.0cm, bb=175 589 365 713]{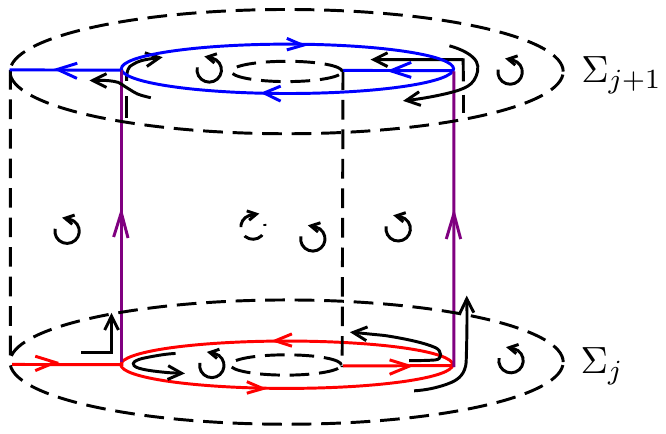}
\caption{Attach $G_j\times [0,1]$. Black arrows in the figure show the smooth directions of the immersion of $S(P)$ at the vertices.}\label{fig38}
\end{center}
\end{figure}

If $w_k=\sigma_{a_p}^{m_p}$ then
we modify the $p$-th blue circle on $\Sigma_{6k+1}$ as shown on the left in Figure~\ref{fig38-1} if $m_p>0$ and as shown on the right if $m_p<0$.
We may verify that all vertices appearing by this modification are of $\ell$-type.

\begin{figure}[htbp]
\begin{center}
\includegraphics[width=12.0cm, bb=139 616 447 713]{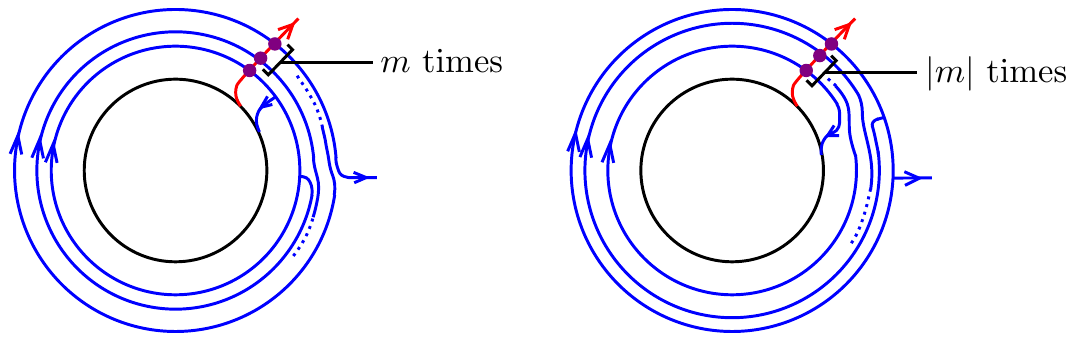}
\caption{Modification of the polyhedron according to the Dehn twists $\sigma_{a_p}^{m_p}$.}
\label{fig38-1}
\end{center}
\end{figure}

If $w_k=\sigma_{b_p}^{m_p}$ then, we apply the same modification to the blue curve on $\Sigma_{6k+2}$ shown on the left in Figure~\ref{fig38-2}, and
if $w_k=\sigma_{c_p}^{m_p}$ then we do the same modification to the blue curve 
on $\Sigma_{6k+5}$ shown on the right. All vertices appearing by these modifications are also of $\ell$-type.

\begin{figure}[htbp]
\begin{center}
\includegraphics[width=8.0cm, bb=194 623 399 711]{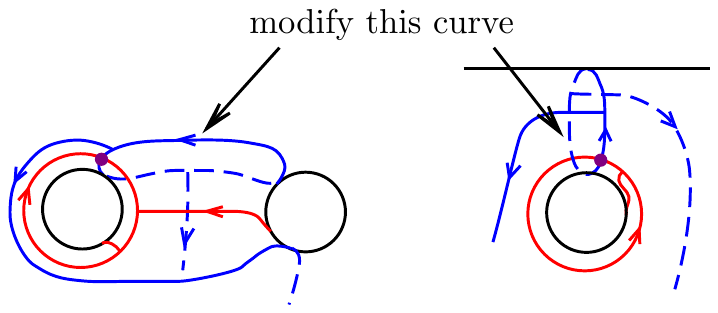}
\caption{Curves for modifications corresponding to the Dehn twists $\sigma_{b_p}^{m_p}$ and $\sigma_{c_p}^{m_p}$.}\label{fig38-2}
\end{center}
\end{figure}

We may choose a monodromy vector field of the open book on $M\setminus \Int\Nbd(L;M)$ such that it is positively transverse to the branched polyhedron constructed above.
We denote the obtained branched polyhedron embedded in $M\setminus\Int \Nbd(L;M)$ by $\hat P_1$.

Next, we construct a branched polyhedron in $\Nbd(L;M)$.
For each $(i,j)$, $i=1,\ldots,2g$, $j=1,\ldots,6n$,
we set a polyhedron $P_{i,j}$ shown in Figure~\ref{fig37}.
Glue the left side of $P_{i,j}$ and the right side of $P_{i+1,j}$ canonically,
where we regard $P_{2g+1,j}$ as $P_{1,j}$,
and denote the obtained polyhedron by $P_j$ for $j=1,\ldots,6n$.
We then glue $P_j$ and $P_{j+1}$ so that the top of $P_{i,j}$ is identified with
the bottom of $P_{i,j+1}$, where $P_{6n+1}$ is regarded as $P_1$.
We denote the obtained polyhedron by $\hat P_2$.
We assign orientations to the regions of $\hat P_2$ as shown in Figure~\ref{fig37},
so that $\hat P_2$ is a branched polyhedron.
The boundary of $\hat P_2$ consists of $2g$ circles and a connected graph.
Attach a disk $D_i$ to each of $2g$ circles and denote the obtained branched polyhedron by $\hat P_3$.
We can embed $\hat P_3$ into the solid torus $D^2\times S^1$ properly.
We regard this $D^2\times S^1$ as $\Nbd(L;M)$ and glue $\hat P_3$ and $\hat P_1$
so that the simple closed curve on $\partial \hat P_3\cap P_j$ (green curve)
coincides with the boundary of $\Sigma_j$ in $\hat P_1$.
We denote the obtained branched polyhedron embedded in $M$ by $\hat P_4$.
We can easily check that all vertices of $\hat P_4$ are of $\ell$-type.

\begin{figure}[htbp]
\begin{center}
\includegraphics[width=12.0cm, bb=130 571 449 712]{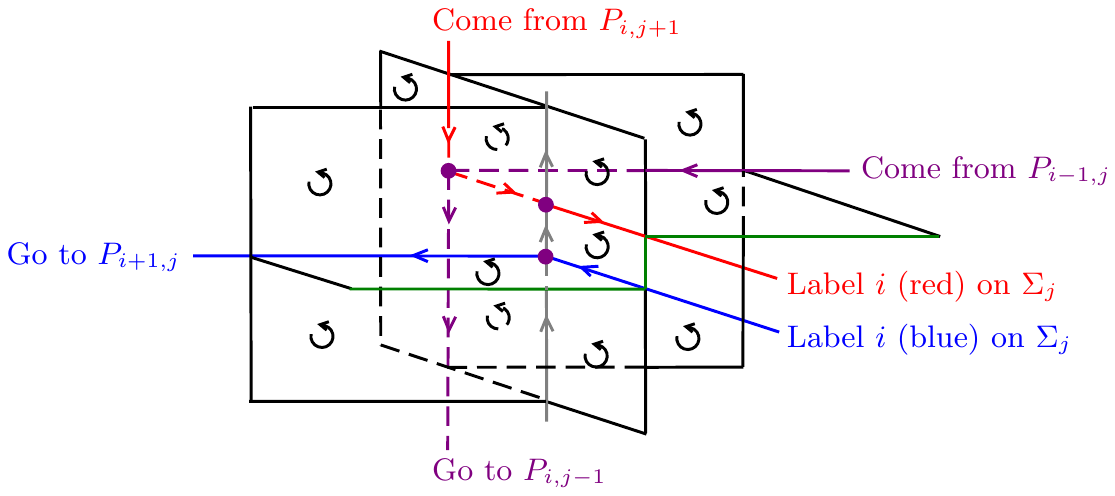}
\caption{The piece $P_{i,j}$.}\label{fig37}
\end{center}
\end{figure}

The complement $M\setminus \hat P_4$ consists of $2g+6n$ open balls
$B_1,\ldots,B_{2g}$ and $B'_1,\ldots, B'_{6n}$, where 
the $2g$ balls $B_1,\ldots, B_{2g}$ are the components bounded by $\hat P_3$
and each $B'_j$, $j=1,\ldots,6n$, 
is the component of  $M\setminus \hat P_4$ containing the connected component of 
$M\setminus\left(\hat P_4\cup \Nbd(L;M)\right)$
between $\Sigma_j$ and $\Sigma_{j+1}$.

The singular set $S(P)$ consists of $2g+6n$ immersed circles
$C_1,\ldots,C_{2g}$ and $C'_1,\ldots, C'_{6n}$,
where $C_i$ is a simple closed curve passing through $P_{i,1},\ldots,P_{1,6n}$,
which is drawn in gray in Figure~\ref{fig37},
and $C_j'$ is a simple closed curve given as follows:
Start from Label $1$ (red) on $\Sigma_{6k}$ in Figure~\ref{fig31} and follow the arrow.
Let $P_{i_0,6k}$ be the piece of $\hat P_4$ that contains the starting point.
The curve passes through $P_{i_0,6k+1}$ as depicted with black arrows in Figure~\ref{fig38} and comes back to Label $1$ (blue) on $\Sigma_{6k+1}$.
Then, as shown in Figure~\ref{fig37}, it goes to Label $2$ on $\Sigma_{6k+1}$ and then
goes down to Label $2$ (red) on $\Sigma_{6k}$.
By the same observation, we see that it passes through Labels $1, 2,\ldots, 8$ (red) on $\Sigma_{6k}$
in this order. This curve goes further and finally comes back to Label $1$ (red) on $\Sigma_{6k}$.
This is the curve $C_1'$. By the same way, the curves $C'_j$, $j=2,\ldots,6n$, are defined.

We modify a part of $C_i$ on $P_{i,1}$ for each $i=1,\ldots,2g$ as shown in Figure~\ref{fig39} and remove the bigon that appears by this move.
This modification does not yield a new vertex, connects $B_i$ to $B'_1$
and makes $C_i$ and $C_1'$ to be one immersed circle.
Applying this modification to each of $C_1,\ldots, C_{2g}$,
we obtain a branched polyhedron $\hat P_5$ with only vertices of $\ell$-type 
such that $M\setminus \hat P_5$ consists of $6n$ open balls $B''_1,B_2',\ldots, B'_{6n}$
and $S(P)$ consists of $6n$ immersed circles $C''_1,C_2'\ldots, C'_n$,
where $B_1''$ is the open ball obtained by connecting $B_1,\ldots,B_{2g}$ with $B_1'$
and $C_1''$ is the immersed circle obtained by connecting $C_1,\ldots,C_{2g}$ with $C_1'$.

\begin{figure}[htbp]
\begin{center}
\includegraphics[width=14.0cm, bb=129 580 562 711]{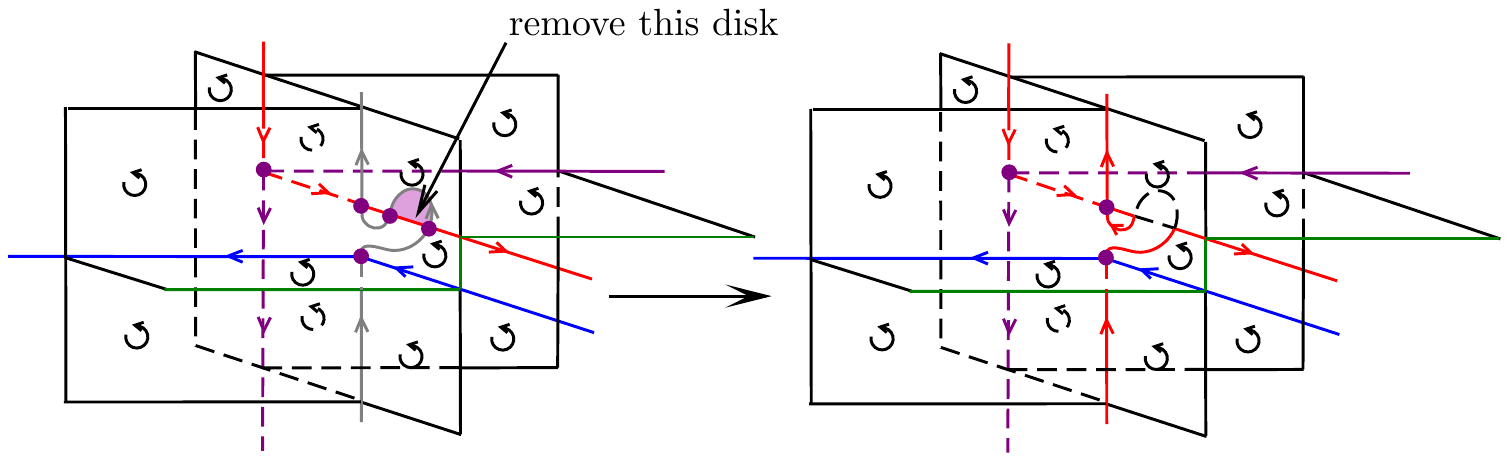}
\caption{Modification of $C_i$ on $P_{i,j}$.}\label{fig39}
\end{center}
\end{figure}

For each $j=2,\ldots,6n$,
we apply the same modification to the same part of $C_1''$ on $P_{1,j}$.
This does not yield a new vertex, connects $B'_j$ to $B''_1$
and makes $C''_1,C_2',\ldots,C_n'$ to be one immersed circle.
The obtained branched polyhedron $\hat P_6$ satisfies that all vertices are of $\ell$-type,
$M\setminus \hat P_6$ is one open ball and $S(P)$ is the image of an immersion of one circle.

We can see directly that 
there exists a monodromy vector field of the open book whose generated flow is carried by $\hat P_6$.
Thus we obtain the assertion.
\end{proof}

To prove Theorem~\ref{thm02}, we need to show the following stronger statement.

\begin{proposition}\label{prop63}
Let $(M, \Sigma, L, \phi)$ be an open book decomposition of 
a closed, connected, oriented $3$-manifold $M$ whose page is a surface of genus at least $2$ and whose binding is connected.
Then there exists a positive flow-spine $P$ of $M$ and a contact form $\alpha$ on $M$ such that the Reeb flow is carried by $P$ and also is a monodromy vector field of the open book.
\end{proposition}


\begin{proof}
By Proposition~\ref{prop61}, there exists a positive flow-spine $\hat P_6$ whose flow is a monodromy vector field of the open book.
In the construction of $\hat P_6$, the polyhedron $\hat P_3$ is contained in $\Nbd(L;M)$. 
Let $N_T$ be a small tubular neighborhood of $\partial \Nbd(L;M)$ in $M\setminus\Int\Nbd(L;M)$. Let $(z', r', \theta')$ be local coordinates on $N_T$ given as shown in Figure~\ref{fig45}. 
We choose coordinates $(u,v,w)$ on 
$N_D=D^2\times [0,1]$
so that $\frac{\partial}{\partial w}$ is tangent to $\partial \Nbd(L;M)$.
Then we take a $1$-form  $\beta$ on $\hat P_6$ and make a contact form $\alpha'$ on $M$ 
as explained in Section~\ref{sec25} such that $\ker\alpha'$ is supported by $\hat P_6$.

\begin{figure}[htbp]
\begin{center}
\includegraphics[width=11.5cm, bb=132 610 438 710]{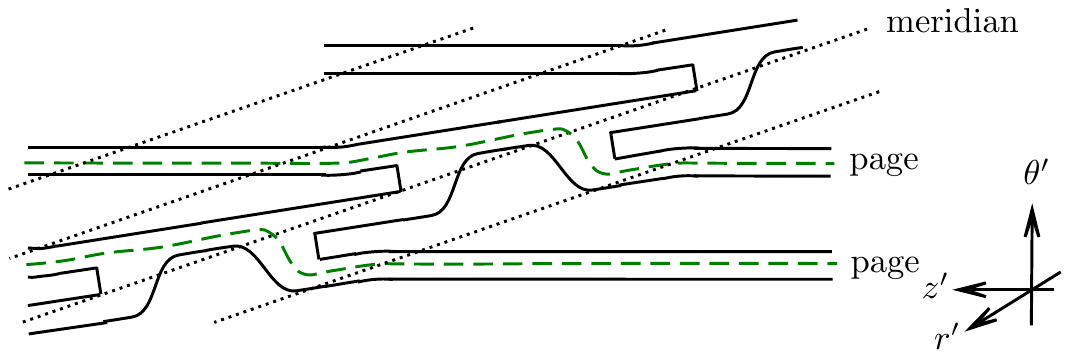}
\caption{Mutual positions of $N_P$, pages and the meridians of the open book on $\partial \Nbd(L;M)$.}\label{fig45}
\end{center}
\end{figure}

We may set the position of the pages of the open book in $M\setminus\Int\Nbd(L;M)$ in such a way that the pages are transverse to $R_{\alpha'}$ since the Reeb flow of $\alpha'$ is carried by $\hat P_6$. 
However, it is not always possible to choose positions of pages of the open book in $\Nbd(L;M)$ so that $R_{\alpha'}$ is transverse to the pages in $\Nbd(L;M)$.
For example, if $R_{\alpha'}$ is winding along $L$ in the direction opposite to the open book then definitely we cannot choose such positions. For this reason, we need to 
reconstruct $\alpha'$ in $\Nbd(L;M)$ such that $R_{\alpha'}$ is transverse to the pages of an open book in $\Nbd(L;M)$.

Let $D^2_s\times S^1$ be the solid torus equipped with the contact form 
$\alpha_b=h_1(r)dz+h_2(r)d\theta$,
where $D^2_s$ is the $2$-disk with radius $s>0$, $(r,\theta)$ are the polar coordinates of $D^2_s$, and $h_1(r)$ and $h_2(r)$ are functions such that 
\begin{itemize}
\item $(h_1(0), h_2(0))=(1,0)$,
\item $\left(\frac{dh_1}{dr}(0), \frac{dh_2}{dr}(0)\right)=(0,1)$,
\item $h_1$ is monotone decreasing and $h_2$ is monotone increasing.
\end{itemize}
The curve $(h_1(r), h_2(r))$ is given as shown in Figure~\ref{fig46}. 
Note that $\alpha_b$ is a contact form
and its Reeb vector field $R_{\alpha_b}$ is parallel to
$-\frac{dh_1}{dr}(r)\frac{\partial}{\partial \theta}+\frac{dh_2}{dr}(r)\frac{\partial}{\partial z}$
in the same direction.

\begin{figure}[htbp]
\begin{center}
\includegraphics[width=8.0cm, bb=173 609 406 711]{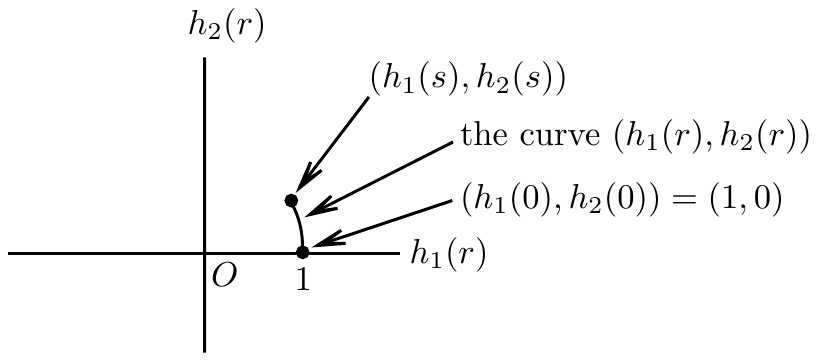}
\caption{The curve $(h_1(r), h_2(r))$.}\label{fig46}
\end{center}
\end{figure}

Let $g_T$ be a diffeomorphism from $N_T'=\{(r,\theta,z)\in D^2_s\times S^1\mid s'\leq r\leq s\}$ to $N_T$ 
that maps $\partial D_{s'} \times \{p\}$, $p\in S^1$, to the meridians of the open book 
and $\{q\}\times S^1$, $q\in\partial D_{s'}$, to the pages in $N_T$, where
$s'$ is a positive real number with $s'<s$ that is sufficiently close to $s$,
and further satisfies that
\begin{itemize}
\item[(i)] $g_T^*\eta=\nu_1(\theta,z)d\theta+\nu_2(\theta,z)dz$ with $\nu_1(\theta,z)>0$ and $\nu_2(\theta,z)\geq 0$, and 
\item[(ii)] $R_{\alpha_b}$ is positively transverse to $\ker (g_T^*\eta)$.
\end{itemize}
We may choose such a diffeomorphism as shown in Figure~\ref{fig45-2}.

\begin{figure}[htbp]
\begin{center}
\includegraphics[width=11.5cm, bb=151 609 444 710]{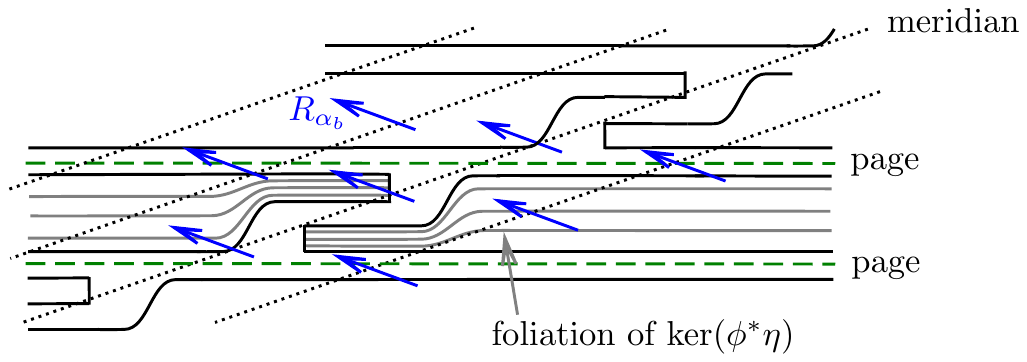}
\caption{Mutual positions of the preimages of $N_P$, pages and the meridians of the open book by $g_T$, the Reeb vector field $R_{\alpha_b}$ (blue arrows) and the foliations of $\ker(g_T^*\eta)$ (drawn in gray) on $\partial \Nbd(L;M)$.}\label{fig45-2}
\end{center}
\end{figure}

Set $\alpha_c=g_T^*(\alpha')$.
We define a contact form $\alpha_T$ on $N_T$ as
\[
   \alpha_T=(1-\chi(r))\alpha_b+\chi(r)\alpha_c,
\]
where $\chi$ is a monotone increasing smooth function which is $0$ for $r\leq s'$ and $1$ for $r\geq s$. Then
\[
\begin{split}
\alpha_T\land d\alpha_T
=&(1-\chi(r))^2\alpha_b\land d\alpha_b+\frac{d\chi}{dr}(r)\alpha_b\land dr\land \alpha_c \\
&+(1-\chi(r))\chi(r)(\alpha_b\land d\alpha_c+\alpha_c\land d\alpha_b)
+\chi(r)^2\alpha_c\land d\alpha_c.
\end{split}
\]

The second term is 
$\frac{d\chi}{dr}(r)\alpha_b\land dr\land \alpha_c=\frac{d\chi}{dr}(r)(h_1(r)dz+h_2(r)d\theta)\land dr\land g_T^*(\hat\beta_M+R\eta)$.
Since  $g_T^*\eta=\nu_1(\theta,z)d\theta+\nu_2(\theta,z)dz$ with $\nu_1(\theta,z)>0$ and $\nu_2(\theta,z)\geq 0$ by the condition~(i), we have 
\[
   (h_1(r)dz+h_2(r)d\theta)\land dr\land g_T^*\eta=
(h_1(r)\nu_1(\theta,z)-h_2(r)\nu_2(\theta,z))dz\land dr\land d\theta.
\]
We choose $s>0$ to be sufficiently small, which implies that $h_2(r)$ is sufficiently small
and hence
$(h_1(r)\nu_1(\theta,z)-h_2(r)\nu_2(\theta,z))dz\land dr\land d\theta>0$.
Choosing $R$ to be sufficiently large, we may conclude that the second term is non-negative.

Now we observe the third term.
Substituting $\alpha_c=g_T^*(\hat\beta_M+R\eta)$,  we have
\[
\alpha_b\land d\alpha_c+\alpha_c\land d\alpha_b=
   \alpha_b\land g_T^*(d\hat\beta_M)+g_T^*(\hat\beta_M+R\eta)\land d\alpha_b,
\]
where we used $d\eta=0$ on $N_D\setminus (\hat W\cup \check W)$ 
in Property~(b) in Section~\ref{sec25}.
Note that $N_T\cap (\hat W\cup \check W)=\emptyset$ due to the choice of the coordinates $(u,v,w)$ in the first paragraph of this proof.
Since $R_{\alpha_b}$ is positively transverse to $\ker\eta$ due to the condition~(ii), 
we have $g_T^*\eta\land d\alpha_b>0$.
Hence the third term is positive for a sufficiently large $R$.
Thus $\alpha_T$ is a contact form on $N_T$.
This means that we can glue the contact forms $\alpha'$ on $M\setminus \Int\Nbd(L;M)$ 
and $\alpha_b$ on $\Nbd(L;M)$ along $N_T=g_T(N_T')$ via the gluing map $g_T$.
We denote the obtained contact form on $M$ by $\alpha''$.

We can extend $\hat P_6$ in $M\setminus \Int\Nbd(L;M)$ into $\Nbd(L;M)$ such that  
the vector field $R_{\alpha''}$ is positively transverse to $\hat P_6$ in $\Nbd(L;M)$.
Thus  the Reeb flow of $\alpha''$ is carried by $\hat P_6$.
We can also extend the pages of the open book in $M\setminus\Int\Nbd(L;M)$ into $\Nbd(L;M)$ with the binding $L=\{0\}\times S^1\subset D_s\times S^1$ such that $R_{\alpha''}$ is transverse to the interior of the pages and tangent to $L$. 
Thus $R_{\alpha''}$ is a monodromy vector field of the open book.
This completes the proof.
\end{proof}

\begin{proof}[Proof of Theorem~\ref{thm02}]
Let $\xi$ be the given contact structure on $M$.
By Theorem~\ref{thm_giroux}, there exist an open book decomposition of $M$ 
and a contact form $\alpha$ on $M$ such that the Reeb vector field $R_\alpha$ 
is a monodromy vector field of the open book.
We may assume that the genus of the page of this open book is at least $2$ and the binding is connected by using plumbings of 
positive Hopf bands.
Note that this modification of the open book does not change the contact structure~\cite{Tor00}.
Let $P$ and $\alpha''$ be the positive flow-spine of $M$ and the contact form on $M$ obtained in the proof of Proposition~\ref{prop63}, where the Reeb flow of $\alpha''$ is carried by $P$ and gives a monodromy vector field of the open book.
Since $\ker\alpha$ and $\ker \alpha''$ are supported by the same open book, there exists a one-parameter family of contact forms from $\ker\alpha$ to $\ker\alpha''$~\cite{Gir02} and hence, 
by Gray's stability~\cite{Gra59}, 
there exists a contactomorphism $\psi$ from $(M,\ker\alpha)$ to $(M,\ker\alpha'')$.
Thus, we have proved that, for the given contact structure $\ker\alpha$, $\psi^*\alpha''$ is a contact form such that $\ker\alpha=\ker\psi^*\alpha''$ and the flow generated by the Reeb vector field $R_{\psi^*\alpha''}$ is carried by the positive flow-spine $\psi^{-1}(\hat P_6)$.
That is, $\ker\alpha$ is supported by $\psi^{-1}(\hat P_6)$. This completes the proof.
\end{proof}

The next is a corollary of Theorem~\ref{thm02}.

\begin{corollary}\label{cor02}
Any flow-spine can be deformed to a positive flow-spine by applying first and second regular moves successively.
\end{corollary}

\begin{proof}
For any homotopy class of non-singular flows, there exists a contact structure whose Reeb flow belongs to the same homotopy class~\cite{Eli89}. 
By Theorem~\ref{thm02}, any contact structure is supported by a positive flow-spine.
Therefore, we can find a sequence of regular moves that relates a given flow-spine with the positive one by Theorem~\ref{thm_ishii}.
\end{proof}

\end{document}